\documentclass[reqno,a4paper]{amsart}
\usepackage{amsmath}
\usepackage{amssymb,amsfonts}
\usepackage{amsthm}
\usepackage{enumerate}
\usepackage{cite}

\usepackage[mathscr]{eucal}
\usepackage{xcolor}

\theoremstyle{plain}
\newtheorem{theorem}{Theorem}[section]
\newtheorem{prop}[theorem]{Proposition}
\newtheorem{lemma}{Lemma}[section]
\newtheorem{corol}{Corollary}[theorem]

\theoremstyle{definition}

\theoremstyle{remark}

\numberwithin{equation}{section}

\begin{document}
	\title[On the Diophantine equation $%
(5pn^{2}-1)^{x}+(p(p-5)n^{2}+1)^{y}=(pn)^{z}$]{On the Diophantine equation $%
	(5pn^{2}-1)^{x}+(p(p-5)n^{2}+1)^{y}=(pn)^{z}$ }
\author{El\.{I}f K{\i}z{\i}ldere and G\"{o}khan Soydan}
\address{\textbf{Elif K{\i}z{\i}ldere}\\
	Department of Mathematics \\
	Bursa Uluda\u{g} University\\
	16059 Bursa, Turkey}
\email{elfkzldre@gmail.com}
\urladdr{https://orcid.org/0000-0002-7651-7001}
\address{\textbf{G\"{o}khan Soydan} \\
	Department of Mathematics \\
	Bursa Uluda\u{g} University\\
	16059 Bursa, Turkey}
\email{gsoydan@uludag.edu.tr }
\urladdr{https://orcid.org/0000-0002-6321-4132}
\thanks{}
\begin{abstract}
	Let $p$ be a prime number with $p>3$, $p\equiv 3\pmod{4}$ and let $n$ be a
	positive integer. In this paper, we prove that the Diophantine equation $%
	(5pn^{2}-1)^{x}+(p(p-5)n^{2}+1)^{y}=(pn)^{z}$ has only the positive integer
	solution $(x,y,z)=(1,1,2)$ where $pn \equiv \pm1 \pmod 5$. As an another result, we
	show that the Diophantine equation $(35n^{2}-1)^{x}+(14n^{2}+1)^{y}=(7n)^{z}$
	has only the positive integer solution $(x,y,z)=(1,1,2)$ where $n\equiv \pm 3%
	\pmod{5}$ or $5\mid n$. On the proofs, we use the properties of Jacobi
	symbol and Baker's method.
\end{abstract}

\maketitle

\section{Introduction}

\label{sec:1} 

Let $\mathbb{N}$ be the set of positive integers, $\mathbb{Z}$ the set of
integers and $\mathbb{Q}$ the field of rational numbers. Suppose that $a,b,c$
are relatively prime positive integers with $\min \{a,b,c\}>1$. The equation 
\begin{equation}
	a^{x}+b^{y}=c^{z},\ \ x,y,z\ \in \ \mathbb{N},  \label{1.1}
\end{equation}%
is an important and general exponential Diophantine equation which attracts
the attention of many authors.

In 1933, the first work was recorded by Mahler \cite{Ma}. He proved
finiteness of solutions of equation \eqref{1.1}. His method is a $p$-adic
analogue of that of Thue-Siegel, so it is ineffective in the sense that it
gives no indication on the number of possible solutions. In 1940, an
effective result for solutions of \eqref{1.1} was given by Gel'fond \cite{Ge}%
, who used Baker's method, which is on lower bounds for linear form in the
logarithms of algebraic numbers. If $x=y=p$, then the equation \eqref{1.1}
becomes the Catalan-Fermat equation $x^p+y^p=z^q$ (see \cite{DM}, \cite{Ell}%
). The case $p=2$ is an important ingredient to the topic. Essentially this
case has been solved in a quite general form by Ellenberg \cite{Ell}.

Some authors determined complete solutions of equation \eqref{1.1} for small
values of $a,b,c$ by using elemantary congruences, the quadratic reciprocity
law and the arithmetic of quadratic (or cubic) fields (see \cite{Ha}, \cite%
{Na} and \cite{U}).

Equation \eqref{1.1} has been considered in detail for Pythagorean numbers $%
a,b,c$ as well. In 1956, Je\'{s}manowicz \cite{J} set a conjecture that if $%
a,b,c$ are Pythagorean triples i.e. positive integers satisfying $%
a^{2}+b^{2}=c^{2}$, then \eqref{1.1} has only the positive integer solution $%
(x,y,z)=(2,2,2)$. Other conjectures related to equation \eqref{1.1} were set
and discussed. One is the extension of Je\'{s}manowicz' conjecture due to
Terai. In 1994, Terai \cite{T1}, conjectured that if equation \eqref{1.1}
has a solution $(x,y,z)=(p,q,r)$ with $\min (p,q,r)>1$, then the equation %
\eqref{1.1} has only the solution $(x,y,z)=(p,q,r)$. But, the years between
1999 and 2011, some authors determined the exceptional cases of Terai's
conjecture on the equation \eqref{1.1} (see \cite{Cao}, \cite{Cao2}, \cite%
{Le2}, \cite{Mi1}, \cite{Mi2}, \cite{T2}). On the other hand, heuristics
show that the equation \eqref{1.1} has at most one solution $(x,y,z)$ with $%
\min (x,y,z)>1$ (see \cite{CM}, \cite{Le3}). The conjecture was proved for
some special cases. But, in general, the problem is still unsolved. For the
details, see the survey paper on Je\'{s}manowicz and Terai conjectures,
which is written by Soydan et al \cite{SCDT}.

Now we consider the Diophantine equation 
\begin{equation}
	(am^{2}+1)^{x}+(bm^{2}-1)^{y}=(cm)^{z}\text{,}  \label{1.2}
\end{equation}%
where $a,b,c,m$ are positive integers such that $a+b=c^{2}$. Recently, there
are many papers discussing the solutions $(x,y,z)$ of the equation %
\eqref{1.2} (see \cite{Ber}, \cite{FY}, \cite{KMS}, \cite{MT1}, \cite{P}, 
\cite{SL}, \cite{T3,TH1,TH2,WWZ}).

In this paper, we are interested in the positive integer solutions of the
Diophantine equation 
\begin{equation}
	(5pn^{2}-1)^{x}+(p(p-5)n^{2}+1)^{y}=(pn)^{z},  \label{1.4}
\end{equation}%
where $p\equiv 3\pmod 4$ and $p>3$ is prime. Our main theorem is the
following.

\begin{theorem}[Main theorem]
	\label{theo:1.1} Suppose that $n$ is a positive integer, $p>3$ is a prime
	number with $p \equiv 3 \pmod 4$ and $pn \equiv \pm1 \pmod 5$. Then the equation \eqref{1.4}
	has only the positive integer solution $(x,y,z)=(1,1,2)$.
\end{theorem}

\begin{corol}
	\label{corol:1.2} Suppose that $n\equiv\pm 3 \pmod{5}$ or $5 \mid n$. Then
	the exponential Diophantine equation 
	\begin{equation}  \label{1.5}
		(35n^2-1)^x+(14n^2+1)^y=(7n)^z
	\end{equation}
	has only the positive integer solution $(x,y,z)=(1,1,2)$.
\end{corol}


\section{Auxiliary results}


In this section, we recall some results that will be important for the
proofs of the main theorem and the corollary. We first need a result on
lower bounds for linear forms in the logarithms of two algebraic numbers to
obtain an upper bound for the solution $y$ of Pillai's equation $w^z-v^y=u$
under some conditions. (We will give the details about Pillai's equation in
Section \ref{S.3.3}). Here, we first present some notations. Suppose that $\beta_{1}$
and $\beta_{2}$ are real algebraic numbers with $|\beta_{1}|\geq 1$ and $%
|\beta_{2}|\geq 1$. Consider the linear form 
\begin{equation*}
	\Omega= b_{2}\log\beta_{2}-b_{1}\log\beta_{1},
\end{equation*}
where $b_{1}$ and $b_{2}$ are positive integers. For any non-zero algebraic
number $\gamma$ of the degree $d$ over $\mathbb{Q}$, whose minimal
polynomial over $\mathbb{Z}$ is $a_{0} \prod_{j=1}^{d} (X-\gamma^{(j)})$, we
denote by 
\begin{equation*}
	h(\gamma)=\frac{1}{d} (\log|a_{0}|+ \sum_{j=1}^{d}
	\log\max(1,|\gamma^{(j)}|))
\end{equation*}
its absolute logarithmic height, where $(\gamma^{(j)})_{1\leq j \leq d}$ are
conjugates of $\gamma$. Suppose that $B_{1}$ and $B_{2}$ are real numbers
greater than 1 with 
\begin{equation*}
	\log B_{i} \geq \max\{h(\gamma_{i}), \frac{|\log \gamma_{i}|}{D}, \frac{1}{%
		D}\},
\end{equation*}
for $i \in \{1, 2\}$, where $D$ is the degree of the number field $%
\mathbb{Q}(\gamma_{1}, \gamma_{2})$ over $\mathbb{Q}$. Define 
\begin{equation*}
	b^{\prime }= \dfrac{b_{1}}{D\log B_{2}}+\dfrac{b_{2}}{D\log B_{1}}.
\end{equation*}
To use Laurent's result in \cite[Corollary 2]{La}, we take $m=10$ and $%
C_{2}=25.2$.

\begin{prop}
	\label{prop:2.1} \textnormal{(Laurent \cite{La})} Suppose that $\Omega$ is
	given as above, with $\gamma_{1}>1$ and $\gamma_{2}>1$. Let $\beta_{1}$ and $%
	\beta_{2}$ be multiplicatively independent. Then 
	\begin{equation*}
		\log |\Omega| \geq -25.2 D^4 (\max \{\log b^{\prime }+0.38, \frac{10}{D}%
		\})^{2}\log B_{1} \log B_{2}.
	\end{equation*}
\end{prop}

Now, we need a result of Bugeaud \cite{Bu}, which is based on linear forms
in $p$-adic logarithms. Here, we take the case where $y_{1}=y_{2}=1$ in the
notation of \cite[p.375]{Bu}.

Suppose that $p$ is a prime and let $\upsilon_{p}$ denote the standard $p$%
-adic valuation normalized by $\upsilon_{p}(p)=1$. Suppose that $\delta_{1}$
and $\delta_{2}$ are non-zero integers prime to $p$ and let $h$ denote the
smallest positive integer such that 
\begin{equation*}
	\upsilon_{p}(\delta_{1}^{h}-1)>0, \ \ \upsilon_{p}(\delta_{2}^{h}-1)>0.
\end{equation*}
Suppose that there exists a real number $F$ such that $\upsilon_{p}(%
\delta_{1}^{h}-1)\geq F> \frac{1}{p-1}$. Bugeaud \cite{Bu} obtained explicit
upper bounds for the $p$-adic valuation of 
\begin{equation*}
	\Omega= \delta_{1}^{b_{1}}- \delta_{2}^{b_{2}},
\end{equation*}
where $b_{1}$ and $b_{2}$ are positive integers. We have

\begin{prop}
	\label{prop:2.2} \textnormal{(Bugeaud \cite{Bu})} Let $B_{1}>1, B_{2}>1$ be
	real numbers such that 
	\begin{equation*}
		\log B_{i} \geq \max \{\log |\delta_{i}|, F\log p\}, \ i=1, 2,
	\end{equation*}
	and we put 
	\begin{equation*}
		b^{\prime }= \dfrac{b_{1}}{\log B_{2}}+\dfrac{b_{2}}{\log B_{1}}.
	\end{equation*}
	If $\delta_{1}$ and $\delta_{2}$ are multiplicatively independent, then we
	have the upper estimates 
	\begin{equation*}
		\upsilon_{p}(\Omega) \leq \dfrac{36.1h}{F^3(\log p)^4} (\max \{\log b'+\log (F\log p)+0.4, 6F\log p, 5 \})^2 \log B_{1} \log B_{2}.
	\end{equation*}
\end{prop}


\section{Proof of Theorem \protect\ref{theo:1.1}}


In this section, we prove Theorem \ref{theo:1.1}.

Suppose that $(x,y,z)$ is a solution of \eqref{1.4}. Considering \eqref{1.4}
modulo $p$ implies that $(-1)^x+1 \equiv 0 \pmod p$. Therefore $x$ is odd.

\subsection{The case where n is odd with $n \not\equiv 0 \pmod 5$}

\begin{lemma}
	\label{lemma:3.2} Let $n$ be odd and $pn \equiv \pm1 \pmod 5$. If the
	equation \eqref{1.4} has a positive integer solution, then $x=1$.
\end{lemma}

\begin{proof}
	Assume that $pn \equiv \pm 1 \pmod 5$. If $x>2$, then we desire to get a
	contradiction. We consider the proof in two cases: Case 1: $n \equiv 3 
	\pmod
	4$, \ Case 2: $n \equiv 1 \pmod 4$.
	
	Case 1: $n \equiv 3 \pmod 4$. Then considering \eqref{1.4} modulo $4$
	implies that $3^{y} \equiv 1 \pmod 4$, then $y$ is even.
	
	Considering \eqref{1.4} modulo $5$, one gets 
	\begin{equation}  \label{3.1}
		-1+(\pm 1) \equiv (\pm 1) \pmod 5
	\end{equation}
	which is a contradiction. Therefore we get $x=1$.
	
	Case 2: $n \equiv 1 \pmod 4$. Then $\Big(\dfrac{5pn^2-1}{p(p-5)n^2+1}\Big)=1$
	and
	
	$\Big(\dfrac{pn}{p(p-5)n^2+1}\Big)=-1$, where $(\frac{*}{*})$ denotes the
	Jacobi symbol. Indeed, 
	\begin{equation*}
		\Big(\dfrac{5pn^2-1}{p(p-5)n^2+1}\Big)=\Big(\dfrac{5pn^2+p(p-5)n^2}{%
			p(p-5)n^2+1}\Big)=\Big(\dfrac{p^{2}n^{2}}{p(p-5)n^2+1}\Big)=1
	\end{equation*}
	and 
	\begin{equation*}
		\Big(\dfrac{pn}{p(p-5)n^2+1}\Big)=\Big(\dfrac{p}{p(p-5)n^2+1}\Big)\Big(%
		\dfrac{n}{p(p-5)n^2+1}\Big)
	\end{equation*}
	\begin{equation*}
		=-\Big(\dfrac{p(p-5)n^2+1}{p}\Big)\Big(\dfrac{p(p-5)n^2+1}{n}\Big)=-1,
	\end{equation*}
	when $n \equiv 1 \pmod 4$ and $p \equiv 3 \pmod 4$. So we get that $z$ is
	even from \eqref{1.4}.
	
	Considering \eqref{1.4} modulo $4$ implies that $3^{y} \equiv (pn)^z \equiv
	3^{z} \equiv 1 \pmod 4$ since $z$ is even. So $y$ is even. Again considering %
	\eqref{1.4} modulo $5$, one obtains 
	\begin{equation*}
		(-1)+ (\pm 1) \equiv 1 \pmod 5
	\end{equation*}
	which is a contradiction. Therefore we get $x=1$.
\end{proof}

\subsection{The case n is even}

\begin{lemma}
	\label{lemma:3.1} Let $n$ be even. Then the equation \eqref{1.4} has only
	the positive integer solution $(x,y,z)=(1,1,2)$.
\end{lemma}

\begin{proof}
	If $z \leq 2$, we obtain $(x,y,z)=(1,1,2)$ by \eqref{1.4}. Therefore we may
	assume that $z \geq 3$. Considering \eqref{1.4} modulo $n^3$ implies that 
	\begin{equation*}
		5pn^{2}x-1+p(p-5)n^{2}y+1 \equiv 0 \pmod {n^{3}},
	\end{equation*}
	so 
	\begin{equation*}
		5px+p(p-5)y \equiv 0 \pmod n
	\end{equation*}
	which is impossible, as $x$ is odd and $n$ is even. Hence, we get that since 
	$n$ is even, the equation \eqref{1.4} has only the positive integer solution 
	$(x,y,z)=(1,1,2)$. This completes the proof of Lemma \ref{lemma:3.1}.
\end{proof}

\subsection{Pillai's equation $w^{z}-v^{y}=u$}\label{S.3.3}

By Lemma \ref{lemma:3.2}, we get $x=1$ from \eqref{1.4}, satisfying $n$ is
odd with $pn\equiv \pm 1\pmod 5$. If $y\leq 2$, then we obtain $y=1$ and $z=2
$ from \eqref{1.4}. So, we may assume that $y\geq 3$. Thus, our theorem is
reduced to solving Pillai's equation 
\begin{equation}
	w^{z}-v^{y}=u  \label{3.2}
\end{equation}%
with $y\geq 3$, where $u=5pn^{2}-1,v=p(p-5)n^{2}+1$ and $w=pn$.

We first want to determine an upper bound for $y$.

\begin{lemma}
	\label{lemma:3.3} $y < 2521\log w$.
\end{lemma}

\begin{proof}
	By \eqref{3.2}, we now consider the following linear form in two logarithms: 
	\begin{equation*}
		\Omega= z\log w-y \log v (>0).
	\end{equation*}
	By the inequality $\log (1+k) < k$ for $k > 0$, we obtain 
	\begin{equation}  \label{3.3}
		0< \Omega= \log (\frac{w^{z}}{v^{y}})= \log (1+\frac{u}{v^{y}}) < \frac{u}{%
			v^{y}}.
	\end{equation}
	So we get 
	\begin{equation}  \label{3.4}
		\log \Omega < \log u-y \log v.
	\end{equation}
	
	On the other hand, using Proposition \ref{prop:2.1}, we want to find a lower
	bound for $\Omega $. We obtain the following inequality 
	\begin{equation}
		\log \Omega \geq -25.2(\max \{\log b'+0.38,10\}^2(\log v)(\log w),  \label{3.5}
	\end{equation}%
	where $b^{\prime }=\dfrac{y}{\log w}+\dfrac{z}{\log v}$. \newline
	
	Note that $v^{y+1}>w^{z}$. Indeed, 
	\newpage
	\begin{equation*}
		v^{y+1}-w^{z}=v(w^{z}-u)-w^{z}=(v-1)w^{z}-uv \geq p(p-5)n^{2}p^{2}n^{2}
	\end{equation*}
	\begin{equation*}
		-(5pn^{2}-1)(p(p-5)n^2+1)>0.
	\end{equation*}
	Thus $b^{\prime }<\dfrac{2y+1}{\log w}$. \newline
	
	Set $N=\frac{y}{\log w}$. Using the inequalities \eqref{3.4} and \eqref{3.5}%
	, we obtain 
	\begin{equation*}
		y\log v < \log u+25.2(\max\{\log (2N+\frac{1}{\log w})+0.38, 10\})^{2}(\log
		v)(\log w).
	\end{equation*}
	Since $\log w= \log (pn) \geq \log 21 > 1$, we get 
	\begin{equation*}
		y\log v< \log u+25.2(\max\{\log (2N+1)+0.38, 10\})^2(\log v)(\log w),
	\end{equation*}
	so 
	\begin{equation*}
		N< 1+25.2(\max\{\log (2N+1)+0.38, 10\})^2
	\end{equation*}
	since $\dfrac{\log u}{(\log v)(\log w)}<1$. Therefore we have $N<2521$.
	Thus, the proof Lemma \ref{lemma:3.3} is completed.
\end{proof}

Next, we want to determine a lower bound for $y$.

\begin{lemma}
	\label{lemma:3.4} $y > n^{2}-2$.
\end{lemma}

\begin{proof}
	Using equation \eqref{3.2}, we obtain the following inequality: 
	\begin{equation*}
		(pn)^{z}\geq 5pn^{2}-1+(p(p-5)n^{2}+1)^{3}>(pn)^{3},
	\end{equation*}%
	since $y\geq 3$. Thus $z\geq 4$. Considering \eqref{3.2} modulo $p^{2}n^{4}$
	implies that 
	\begin{equation*}
		5pn^{2}-1+p(p-5)n^{2}y+1\equiv 0\pmod {p^{2}n^{4}}.
	\end{equation*}%
	This shows that $5+(p-5)y\equiv 0\pmod {pn^{2}}$. Thus one gets 
	\begin{equation*}
		y\geq \dfrac{pn^{2}-5}{p-5}=\frac{p}{p-5}n^{2}-\frac{5}{p-5}>n^{2}-2,
	\end{equation*}%
	which completes the proof.
\end{proof}

Now we are ready to prove Theorem \ref{theo:1.1}. By Lemmas \ref{lemma:3.3}
and \ref{lemma:3.4}, we obtain 
\begin{equation}  \label{3.6}
	n^{2}-2 < 2521 \log (pn).
\end{equation}

We want to determine an upper bound for $p$ and then one for $n$. Firstly,
we will show that if $n \equiv 1 \pmod 4$, then $p<6307$. We proved that $z$
is even when $n \equiv 1 \pmod 4$. So set $z=2t$ with $t$ positive integer.
Then, we determine the equation \eqref{3.2} as follows: 
\begin{equation*}
	(w^{2})^{t}-v^{y}=w^{2}-v.
\end{equation*}
Thus $y \geq t$. If $y=t$, then we get $y=t=1$. If $y>t$, then we consider a
\textquotedblleft gap\textquotedblright between the trivial solution $(y,
t)=(1, 1)$ and (possibly) another solution $(y,t)$.

Using $u+v=w^{2}$ and $u+v^{y}=w^{2t}$, we consider the following two linear
forms in two logarithms: 
\begin{equation*}
	\Omega_{0}=2\log w- \log v (>0), \ \ \Omega=2t \log w-y \log v (>0).
\end{equation*}
Hence 
\begin{equation*}
	y \Omega_{0}-\Omega= 2(y-t) \log w \geq 2 \log w,
\end{equation*}
which implies that 
\begin{equation*}
	y>\dfrac{2}{\Omega_{0}} \log w.
\end{equation*}
Using Lemma \ref{lemma:3.3}, we obtain $\dfrac{2}{\Omega_{0}} \log w < 2521
\log w$. So, 
\begin{equation*}
	\frac{2}{2521} < \Omega_{0}= \log (\dfrac{w^{2}}{v})= \log (1+\dfrac{u}{v})<%
	\dfrac{u}{v}=\dfrac{5pn^{2}-1}{p(p-5)n^{2}+1}< \dfrac{5pn^{2}}{p(p-5)n^{2}}
\end{equation*}
\begin{equation*}
	=\dfrac{5}{p-5}.
\end{equation*}
From here, one gets $p<6307$. And hence by \eqref{3.6}, we obtain $n \leq 187$. Secondly, we will show that if $n \equiv 3 \pmod 4$, then $p<12610$. We proved that $y$ is even when $n \equiv 3 \pmod 4$ and we know that $y>1$. If $z\geq 2y$, then by \eqref{3.2} we get
$$
5pn^2-1=w^z-v^y\geq w^{2y}-v^y>w^2-v=p^2n^2-(p^2n^2-5pn^2+1)=5pn^2-1,
$$
which is a contradiction. So, $2y>z$. 

Using $u+v=w^2$ and $u+v^y=w^z$, we consider the following two linear forms in two logarithms:
$$
\lambda_0=2\log w-\log v(>0), \quad \lambda=z\log w-y\log v(>0).
$$
Therefore,
$$
y\lambda_0-\lambda=(2y-z)\log w \geq \log w,
$$
which implies that
$$
y>\dfrac{1}{\lambda_0}\log w.
$$
By Lemma \ref{lemma:3.3}, we get $\dfrac{1}{\lambda_0}\log w<2521\log w$. Thus,
$$
\dfrac{1}{2521}<\lambda_0= \log (\dfrac{w^{2}}{v})= \log (1+\dfrac{u}{v})<
\dfrac{u}{v}=\dfrac{5pn^{2}-1}{p(p-5)n^{2}+1}< \dfrac{5pn^{2}}{p(p-5)n^{2}}
$$
$$
=\dfrac{5}{p-5}.
$$	
From here, one obtains $p<12610$. And hence by \eqref{3.6}, we have $n\leq 192$.

By \eqref{3.3}, we obtain the inequality 
\begin{equation*}
	\Big| \dfrac{\log v}{\log w}-\dfrac{z}{y} \Big| < \dfrac{u}{yv^{y} \log w},
\end{equation*}
which implies that $\Big| \dfrac{\log v}{\log w}-\dfrac{z}{y} \Big|<\dfrac{1%
}{2y^{2}}$ since $y \geq 3$. So $\frac{z}{y}$ is a convergent of the
continued fraction expansion of $\frac{\log v}{\log w}$.

On the other hand, if $\frac{p_{r}}{q_{r}}$ is the $r$-th such convergent,
then 
\begin{equation*}
	\Big| \dfrac{\log v}{\log w}-\dfrac{p_{r}}{q_{r}} \Big| > \dfrac{1}{%
		(u_{r+1}+2)q_{r}^{2}},
\end{equation*}
where $u_{r+1}$ is the $(r+1)$-st partial quotient to $\frac{\log v}{\log w}$%
(see e.g. Khinchin \cite{K}). Set $\frac{z}{y}=\frac{p_{r}}{q_{r}}$. Note
that $q_{r} \leq y$. Then it follows that 
\begin{equation*}
	u_{r+1} > \dfrac{v^{y}\log w}{uy}-2 \geq \dfrac{v^{q_{r}} \log w}{uq_{r}}-2.
\end{equation*}
Finally, running MAGMA \cite{BC} for each $p<12610$ with $p \equiv 3 \pmod 4$, it is
seen that the above inequality is not satisfied for any $r$ with $q_{r} <
2521 \log (pn)$ in the range $1 \leq n \leq 192$. Thus, the proof of Theorem %
\ref{theo:1.1} is completed.


\section{Proof of Corollary \protect\ref{corol:1.2}}


Suppose that $(x,y,z)$ is a solution of \eqref{1.5}. By Theorem \ref%
{theo:1.1}, we may assume that $n \equiv 0 \pmod 5$. We know that $x$ is
odd. Here, we use Proposition \ref{prop:2.2}. So, put $p := 5,\delta _{1} :=
14n^{2}+1,\delta _{2} := 1-35n^{2},b_{1} := y,b_{2} := x$ and 
\begin{equation*}
	\Omega := (14n^{2}+1)^{y}-(1-35n^{2})^{x}.
\end{equation*}
Then we may take $h=1, F=2, B_{1}= 14n^{2}+1, B_{2}=35n^2-1$. Therefore, we
obtain 
\begin{equation*}
	z \leq \dfrac{36.1}{8(\log 5)^{4}}(\max \{\log b'+\log (2\log 5)+0.4, 12\log 5, 5 \})^2 
\end{equation*}
\begin{equation*}
	\times \log (14n^{2}+1)\log (35n^{2}-1),
\end{equation*}
where $b^{\prime }:=\dfrac{y}{\log (35n^{2}-1)}+\dfrac{x}{\log (14n^{2}+1)}$%
. Assume that $z \geq 4$. We show that this will lead to a contradiction.
Considering \eqref{1.5} modulo $n^{4}$, we get 
\begin{equation*}
	35x+14y \equiv 0 \pmod {n^{2}}.
\end{equation*}
Particularly, we obtain $N:= \max \{x, y\} \geq \frac{n^{2}}{49}$. Hence,
since $z \geq N$ and $b^{\prime }\leq \frac{N}{\log n}$, we find 
\begin{equation}  \label{4.1}
	N\leq\dfrac{36.1}{8(\log 5)^4} (\max \{\log (\dfrac{N}{\log n})+ \log (2\log
	5)+0.4,12\log 5\})^{2}
\end{equation}
\begin{equation*}
	\times\log(14n^2+1) \log(35n^2-1).
\end{equation*}

If $n\geq 173356$, then 
\begin{equation*}
	N\leq 0.68(\log (\dfrac{N}{\log n})+\log (2\log 5)+0.4)^{2}\log
	(14n^{2}+1)\log (35n^{2}-1).
\end{equation*}%
When $n^{2}\leq 49N$, from the above inequality, we have 
\begin{equation*}
	N\leq 0.68(\log N-\log (\log (173356))+1.57)^{2}\log (686N+1)\log (1715N-1).
\end{equation*}%
Therefore, we get $N\leq 13732$, which contradicts the fact that $N\geq
n^{2}/49\geq 6.133123007\times 10^{8}$.

If $n<173356$, then using inequality \eqref{4.1}, we obtain 
\begin{equation*}
	\dfrac{n^{2}}{49} \leq \dfrac{36.1}{8(\log 5)^{4}} (12\log 5)^{2} \log
	(14n^{2}+1) \log (35n^{2}-1),
\end{equation*}
namely, 
\begin{equation*}
	\dfrac{n^{2}}{49} \leq 250.8 \log (14n^{2}+1) \log (35n^{2}-1).
\end{equation*}

From the above inequality, we find $n\leq 2031$. Therefore all $x$, $y$ and $%
z$ are also bounded. Using MAGMA \cite{BC}, we see that the eq. \eqref{1.5} has no
solution $(n,x,y,z)$ with $5\mid n$. So, we conclude $z\leq 3$. In this
case, we can easily show that $(x,y,z)=(1,1,2)$. Thus, the proof of
Corollary \ref{corol:1.2} is completed.


\subsection*{Acknowledgements}

We would like to thank to the referees for carefully reading our paper and
for giving such constructive comments which substantially helped improving
the quality of the paper and Prof. Dr. L\'{a}szlo Szalay for some
computational helps. This work was supported by T\"{U}B\.{I}TAK (the Scientific and Technological Research Council of Turkey) under Project No:117F287.


\end{document}